\documentclass[letter paper, 12pt]{article}
\usepackage[english]{babel}
\usepackage[utf8]{inputenc}
\usepackage{amsmath}
\usepackage{amsfonts}
\usepackage{amssymb}

\usepackage[margin=1in]{geometry}
\usepackage{biblatex}

\usepackage{textcomp,gensymb}
\usepackage{hyperref}
\usepackage{float}
\usepackage{epsfig}
\usepackage{multirow}
\usepackage{marginnote}
\usepackage{color}
\usepackage{upgreek}
\newtheorem{theorem}{Theorem}
\newtheorem{lemma}[theorem]{Lemma}

\newcommand{\mbf}[1]{\mathbf{#1}}
\newcommand{\norm}[1]{\left\lVert#1\right\rVert}

\newcommand{\noin}{\noindent}

\newcommand{\qed}{\ \hfill \rule{1ex}{1ex}\\}
\newenvironment{proof}{{\noin \bf Proof}: }{\qed}

\begin{document}
    \title{Improved Bound for the Gerver-Ramsey Collinearity Problem}
    \author{Thomas F. Lidbetter}


    \maketitle
    
    \begin{abstract}
        Let $S$ be a finite subset of $\mathbb{Z}^n$. A vector sequence $(\mbf{z}_i)$ is an $S$-walk if and only if $\mbf{z}_{i+1} - \mbf{z}_i$ is an element of $S$ for all $i$. Gerver and Ramsey showed in 1979 that for $S\subset \mathbb{Z}^3$ there exists an infinite $S$-walk in which no $5^{11} + 1=48{\small,}828{\small,}126$ points are collinear. Here, we use the same general approach, but with the aid of a computer search, to improve the bound to $189$.
    \end{abstract}
    
    
    \section{Introduction}

    Forty-four years ago, Gerver \cite{gerver1979long} and Gerver and Ramsey \cite{gerver1979certain} considered a problem of finding long sequences of vectors avoiding too many collinear points, where the differences between consecutive vectors are drawn from a finite set. To be precise, an $S$-walk is any (finite or infinite) sequence of vectors in $\mathbb{R}^n$, say $(\mbf{z}_i)$, such that $\mbf{z}_{i+1}-\mbf{z}_i\in S$, for all $i$. The points of an $S$-walk are the endpoints of the vectors in the walk. A subset of points from an $S$-walk are collinear if there is a straight line intersecting all points in that subset. Gerver and Ramsey \cite{gerver1979certain} showed that in the case where $S\subset \mathbb{Z}^2$, for every $S$ and positive integer $K$ there exists a finite integer, $N(S,K)$, the largest value such that there are no $K+1$ collinear points in the first $N(S,K)$ terms of the $S$-walk. As an example, for $S=\{\mbf{i},\mbf{j}\}\subset\mathbb{Z}^2$, where $\mbf{i}$ and $\mbf{j}$ are orthonormal unit vectors, the sequence of steps: $\mbf{i},\mbf{j},\mbf{i}$, gives points $(0,0), (1,0), (1,1), (2,1)$. This avoids $3$ collinear points, but adding another step from $S$ will guarantee $3$ collinear points. Similarly, the sequence of steps: $\mbf{i},\mbf{j},\mbf{i},\mbf{i},\mbf{j},\mbf{i},\mbf{i},\mbf{j}$, avoids $4$ collinear points, but adding another step will guarantee $4$ collinear points. This walk is demonstrated in Figure~\ref{fig:maximalwalk}. Gerver \cite{gerver1979long} gives a construction where for any $S$ with at least two linearly independent vectors, $N(S,K)$ grows faster than every polynomial function of $K$. The sequence $a(n)$, the smallest integer $t$ such that every $\{\mbf{i},\mbf{j}\}$-walk of length $t$ is guaranteed to have at least $n$ collinear points, is given in the {\it On-Line Encyclopedia of Integer Sequences} \cite{oeis} as sequence \href{https://oeis.org/A231255}{A231255}. Only the first six terms are known.

    \begin{figure}
        \centering
        \includegraphics[width=0.7\textwidth]{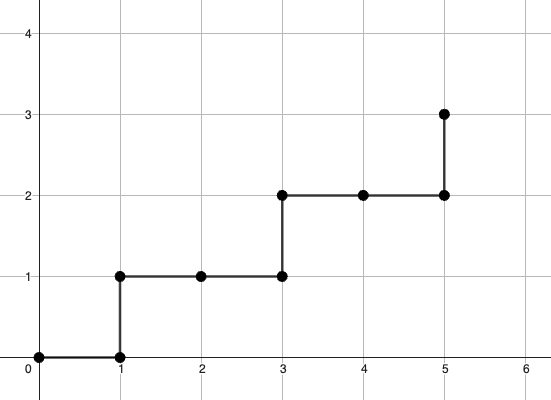}
        \caption{Points in a longest $S$-walk avoiding $4$ collinear points for $S=\{\mbf{i},\mbf{j}\}$.}
        \label{fig:maximalwalk}
    \end{figure}
    
    For the relaxation of the condition on $S$ to $S\subset \mathbb{Z}^3$, Gerver and Ramsey proved \cite{gerver1979certain} that there is an infinite $S$-walk, $W$, that has no $5^{11}+1=48{\small,}828{\small,}126$ collinear points. The approach used is constructive. While the discovery of a finite upper bound on the number of collinear points in an $S$-walk for $S\subset \mathbb{Z}^3$ is an intriguing result, the upper bound is large relative to the value of three that Gerver and Ramsey \cite{gerver1979certain} suggest is the true maximum number of collinear points in $W$. Furthermore, this bound from forty-four years ago has not been improved upon in the time since. We begin here by stating the main result of the present work: an improvement to the bound.
    
    \begin{theorem}
        The infinite $S$-walk, $W$, has no $189$ collinear points.
        \label{thm:collinearbound}
    \end{theorem}
    
    To improve the bound we use the same construction, but we additionally show how to generate $W$ as the fixed point of a morphism to aid in some proofs. This construction is covered in Section~\ref{sec:construction}. In Section~\ref{sec:collinearbound} we follow a similar argument to that of Gerver and Ramsey \cite{gerver1979certain} to assert the improved bound on the number of collinear points. Some case-checking is carried out by a computer program. The algorithms used are described in Section~\ref{sec:algorithms}. An implementation of these algorithms is provided at \url{github.com/FinnLidbetter/avoiding-collinearity} \cite{Lidbetter_Avoiding_Collinearity_2022}. We conclude with some remarks on how this result might be pushed further to reduce the bound in Section~\ref{sec:improving}. Before proceeding, we define some useful terminology and notation.

    \subsection{Terminology and notation}
    
    An alphabet is a set of symbols. A word $w=w_1\cdots w_n$ is a sequence of symbols for symbols $w_1,w_2,\dots,w_n$ in some alphabet $\Sigma$. Define $w[i]$ to be the  $i^{\text{th}}$ symbol of $w$. That is $w[1]=w_1,\dots, w[n]=w_n$ and define $w[i:j]$ to be the subword of $w$ from index $i$ to index $j$ inclusive. The reversal of word $w$ is given by $w^R=w_nw_{n-1}\cdots w_2w_1$.
    
    A morphism $\mu: \Sigma\rightarrow \Sigma^*$ is a mapping from symbols in alphabet $\Sigma$ to strings in $\Sigma^*$. The mapping is extended to $\mu: \Sigma^*\rightarrow \Sigma^*$ by defining $\mu(\varepsilon)=\varepsilon$ and $\mu(uv)=\mu(u)\mu(v)$ for strings $u,v\in\Sigma^*$. For a word $u\in\Sigma^*$ define $\mu^1(u)=\mu(u)$ and $\mu^2(u)=\mu(\mu(u))$. We can then inductively define $\mu^n(u)=\mu(\mu^{n-1}(u))$. We take $\mu^\omega (u)$ to be the fixed point of iterating $\mu$, if such a fixed point exists.
    
    Adopting the same notation as Gerver and Ramsey \cite{gerver1979certain}, for a vector $\mbf{z}\in \mathbb{R}^3$ with $\mbf{z}=z_1{\bf i} + z_2{\bf j} + z_3{\bf k}$ we use $$\norm{\mbf{z}}^\parallel = z_1 + z_2 + z_3$$ and $$\norm{\mbf{z}}^\perp = \sqrt{z_1^2 + z_2^2 + z_3^2 - z_1z_2 - z_2z_3 - z_3z_1}.$$ So $\norm{\mbf{z}}^\parallel$ and $\norm{\mbf{z}}^\perp$ are proportional to the components of $\mbf{z}$ parallel and perpendicular, respectively, to the vector $\mbf{i} + \mbf{j} + \mbf{k}$. For future use we define the constant $\gamma$ to be the length of the component of $\mbf{i},\mbf{j},$ or $\mbf{k}$ perpendicular to $\mbf{i}+\mbf{j}+\mbf{k}$. Then $\gamma=\sqrt{2/3}$ and in general the perpendicular component of $\mbf{z}$ has length $\gamma\norm{\mbf{z}}^\perp$.
    
    \section{The construction}
    \label{sec:construction}
    
    We first recap the method used by Gerver and Ramsey \cite{gerver1979certain}, before giving an alternative method using a morphism for producing the same $S$-walk. The morphism representation makes some later proofs easier.
    
    \subsection{The Gerver and Ramsey construction}\label{sec:GerverRamseyConstruction}
    If $A=(a_1, \dots, a_n)$ and $B=(b_1, \dots, b_m)$ are ordered sets of vectors, and $\beta$ is a vector operator, let $RA=(a_n, \dots, a_1)$, and $(A,B)=(a_1,\dots, a_n,b_1,\dots, b_m)$, and $\beta A=(\beta a_1, \dots, \beta a_n)$. Define vector operators $\alpha$ and $\beta$ that operate on the three orthonormal unit vectors $\bf{i}, \bf{j}, \bf{k}$ as 
    \begin{align*}
        \alpha \bf{i} &=\bf{j} \\
        \alpha \bf{j} &=\bf{i} \\
        \alpha \bf{k} &=\bf{k} \\
        \beta \bf{i} &=\bf{i} \\
        \beta \bf{j} &=\bf{k} \\
        \beta \bf{k} &=\bf{j}
    \end{align*}
    Let $A_0=(\bf{i})$, then define $A_{n+1}=(A_n, \alpha A_n, R\beta A_n, A_n, R\beta\alpha A_n, R \beta A_n, A_n)$. For every positive integer, $n$, define the sequence of vectors $(\mbf{v}_p)_{1\leq p \leq 7^n}$ such that $(\mbf{v}_1,\dots, \mbf{v}_{7^n})=A_n$. We then define $\mbf{z}_p=\Sigma_{q=1}^{p} \mbf{v}_q$ for all positive integers $p$ and take $\mbf{z}_0$ to be the zero vector. Then $W=(\mbf{z}_p)_{p\geq 0}$ is an $S$-walk. This walk begins with the following $35$ steps, i.e., the vectors $\mbf{v}_1,\dots,\mbf{v}_{35}$:
    $$\mbf{i},\mbf{j},\mbf{i},\mbf{i},\mbf{k},\mbf{i},\mbf{i},\mbf{j},\mbf{i},\mbf{j},\mbf{j},\mbf{k},\mbf{j},\mbf{j},\mbf{i},\mbf{i},\mbf{j},\mbf{i},\mbf{i},\mbf{k},\mbf{i},\mbf{i},\mbf{j},\mbf{i},\mbf{i},\mbf{k},\mbf{i},\mbf{i},\mbf{k},\mbf{k},\mbf{j},\mbf{k},\mbf{k},\mbf{i},\mbf{k}.$$ 
    Gerver and Ramsey proved that $W$ is an infinite $S$-walk in which no $5^{11}+1$ points are collinear \cite{gerver1979certain}.
    
    \subsection{A construction using a morphism}

    We first state the morphism for the construction, then give an intuition for why this produces the same sequence, before formally proving that the sequence is identical to that of Section~\ref{sec:GerverRamseyConstruction}. The definition of the morphism is due to Luke Schaeffer, shared via an unpublished personal communication \cite{Schaeffer2012}.
    
    For the alphabet $\Sigma=\{i,j,k,i',j',k',i_b,j_b,k_b,i_b',j_b',k_b'\}$ define the morphism $\mu: \Sigma^*\rightarrow \Sigma^*$ as:
    \begin{align*}
            \mu(i) &= i ~ j' ~ i_b' ~ i ~ k_b ~ i_b' ~ i \\
            \mu(j) &= j ~ k' ~ j_b' ~ j ~ i_b ~ j_b' ~ j \\
            \mu(k) &= k ~ i' ~ k_b' ~ k ~ j_b ~ k_b' ~ k \\
            \mu(i') &= i' ~ k ~ i_b ~ i' ~ j_b' ~ i_b ~ i' \\
            \mu(j') &= j' ~ i ~ j_b ~ j' ~ k_b' ~ j_b ~ j' \\
            \mu(k') &= k' ~ j ~ k_b ~ k' ~ i_b' ~ k_b ~ k' \\
            \mu(i_b) &= i_b ~ i' ~ k ~ i_b ~ i' ~ j_b' ~ i_b \\
            \mu(j_b) &= j_b ~ j' ~ i ~ j_b ~ j' ~ k_b' ~ j_b \\
            \mu(k_b) &= k_b ~ k' ~ j ~ k_b ~ k' ~ i_b' ~ k_b \\
            \mu(i_b') &= i_b' ~ i ~ j' ~ i_b' ~ i ~ k_b ~ i_b' \\
            \mu(j_b') &= j_b' ~ j ~ k' ~ j_b' ~ j ~ i_b ~ j_b' \\
            \mu(k_b') &= k_b' ~ k ~ i' ~ k_b' ~ k ~ j_b ~ k_b'          
    \end{align*}
    Additionally, we define the output map $\phi: \Sigma\rightarrow \{\mbf{i}, \mbf{j}, \mbf{k}\}$ as 
    \begin{align*}
        \phi(i)&=\phi(i_b)=\phi(i')=\phi(i_b')=\mbf{i}\\
        \phi(j)&=\phi(j_b)=\phi(j')=\phi(j_b')=\mbf{j}\\
        \phi(k)&=\phi(k_b)=\phi(k')=\phi(k_b')=\mbf{k}.
    \end{align*}
    We extend this as $\phi: \Sigma^*\rightarrow \{\mbf{i}, \mbf{j}, \mbf{k} \}^*$ by taking $\phi(\varepsilon)=\varepsilon$ and $\phi(uv)=\phi(u)\phi(v)$ for $u,v\in\Sigma^*$.
    
    Observe that since for every $a\in\Sigma$ we have that $\mu(a)=au$ for some $u\in\Sigma^*$, the fixed point $\mu^\omega(i)$ exists and $\mu^n(i)$ is a prefix of $\mu^{n+1}(i)$ for all positive integers, $n$. For convenience, we define $\lambda=\mu^\omega(i)$. Define the vector sequence $(\mbf{u}_p)_{p\geq 1}$ with $\mbf{u}_p = \phi(\lambda[p])$, and then as above, define $(\mbf{x}_p)_{p\geq 0}$ with $\mbf{x}_p=\Sigma_{q=1}^{p} \mbf{u}_q$ for all positive integers $p$ and take $\mbf{x}_0$ to be the zero vector. Then $W_\mu=(\mbf{x}_p)_{p\geq 0}$ is an $S$-walk.  

    \subsection{Equivalence of the constructions}
    
    Th morphism representation is derived from the vector operator construction by observing that the reversal and $\alpha$ and $\beta$ operators can be interpreted as group actions on the set of unit vectors. Each of the $12$ symbols in $\Sigma$ correspond to the $12$ possible group elements. The image of $\mu$ for each of those symbols is the result of applying the group actions corresponding to the operators in the definition of $A_{n+1}$. The group is isomorphic to the dihedral group of order $12$. Considering just the $\alpha$ and $\beta$ operators we can map $\mbf{i},\mbf{j},\mbf{k}$ to each permutation of those vectors, e.g., with the chained operators: $\alpha\alpha$, $\alpha$, $\beta$, $\beta\alpha$, $\alpha\beta$, and  $\beta\alpha\beta$. This gives us the symmetric group $S_3$.  Since the reversal operator commutes with the other operators and it is its own inverse, adding the reversal operator results in the product of $S_3$ and the cyclic group of order $2$. The direct product of these groups is the dihedral group of order $12$. 

    To prove the equivalence of the constructions, we define a few useful functions for translating between the two representations. As an abuse of notation for the purpose of exposition, we overload the definitions of $\alpha$ and $\beta$ to also represent functions $\alpha: \Sigma\rightarrow \Sigma$ and $\beta: \Sigma\rightarrow \Sigma$, where
    \begin{center}
    \begin{tabular}{c c c}
        $\alpha(i) = j'$, & $\alpha(j) = i'$, & $\alpha(k) = k'$, \\
        $\alpha(i') = j$, & $\alpha(j') = i$, & $\alpha(k') = k$, \\
        $\alpha(i_b) = j_b'$, & $\alpha(j_b) = i_b'$, & $\alpha(k_b) = k_b'$, \\
        $\alpha(i_b') = j_b$, & $\alpha(j_b') = i_b$, & $\alpha(k_b') = k_b$, \\
    \end{tabular}
    \end{center}
    and
    \begin{center}
    \begin{tabular}{c c c}
        $\beta(i) = i'$, & $\beta(j) = k'$, & $\beta(k) = j'$, \\
        $\beta(i') = i$, & $\beta(j') = k$, & $\beta(k') = j$, \\
        $\beta(i_b) = i_b'$, & $\beta(j_b) = k_b'$, & $\beta(k_b) = j_b'$, \\
        $\beta(i_b') = i_b$, & $\beta(j_b') = k_b$, & $\beta(k_b') = j_b$. \\
    \end{tabular}
    \end{center}
    We also define a reversal parity switching function $R:\Sigma \rightarrow \Sigma$ as
    \begin{center}
        \begin{tabular}{c c c}
            $R(i) = i_b$, & $R(j) = j_b$, & $R(k) = k_b$, \\
            $R(i') = i_b'$, & $R(j') = j_b'$, & $R(k') = k_b'$, \\
            $R(i_b) = i$, & $R(j_b) = j$, & $R(k_b) = k$, \\
            $R(i_b') = i'$, & $R(j_b') = j'$, & $R(k_b') = k'$.
        \end{tabular}
    \end{center}
    We extend these functions to $\alpha: \Sigma^*\rightarrow \Sigma^*$, $\beta: \Sigma^*\rightarrow \Sigma^*$, and $R:\Sigma^*\rightarrow \Sigma^*$ to operate on words as $\alpha(\varepsilon)=\beta(\varepsilon)=R(\varepsilon)=\varepsilon$ and $\alpha(uv)=\alpha(u)\alpha(v)$, $\beta(uv)=\beta(u)\beta(v)$, and $R(uv)=R(u)R(v)$ for $u,v\in\Sigma^*$.
    
    Lastly, we define function, $g:\{\mbf{i}, \mbf{j}, \mbf{k}\}\times \{0,1\}\times \{0,1\} \rightarrow \Sigma$, mapping $\mbf{a}_p,b_p,c_p$ triples to elements of $\Sigma$:
    \begin{center}
        \begin{tabular}{c c c}
             $g(\mbf{i}, 0, 0) = i$, & $g(\mbf{j}, 0, 0) = j$, & $g(\mbf{k}, 0, 0) = k$, \\
             $g(\mbf{i}, 0, 1) = i'$, & $g(\mbf{j}, 0, 1) = j'$, & $g(\mbf{k}, 0, 1) = k'$, \\
             $g(\mbf{i}, 1, 0) = i_b$, & $g(\mbf{j}, 1, 0) = j_b$, & $g(\mbf{k}, 1, 0) = k_b$, \\
             $g(\mbf{i}, 1, 1) = i_b'$, & $g(\mbf{j}, 1, 1) = j_b'$, & $g(\mbf{k}, 1, 1) = k_b'$. \\
        \end{tabular}
    \end{center}
    One can verify the following identities using these functions and vector operators, keeping in mind that the $\alpha$ and $\beta$ on the left side are functions from $\Sigma$ to $\Sigma$, whereas the the $\alpha$ and $\beta$ on the right side are vector operators:
    \begin{align}
        \alpha(g(\mbf{a},b,c)) & = g(\alpha \mbf{a}, b, 1-c), \\
        \beta(g(\mbf{a},b,c)) & = g(\beta \mbf{a}, b, 1 - c), \\
        R(g(\mbf{a},b,c)) & = g(\mbf{a}, 1-b, c).
    \end{align}

    By the definition of $\mu(i)$, we have $$\mu^{n+1}(i)=\mu^n(i)\mu^n(j')\mu^n(i_b')\mu^n(i)\mu^n(k_b)\mu^n(i_b')\mu^n(i).$$ 
    It is straightforward to prove:
    \begin{align}
        \mu^n(j')&=\alpha(\mu^n(i)), \label{eqn:j'} \\
        \mu^n(i_b')&=R(\beta(\mu^n(i))^R), \label{eqn:ib'} \\
        \mu^n(k_b)&=R(\beta(\alpha(\mu^n(i)))^R). \label{eqn:kb}
    \end{align} 
    Substituting each of these into the statement of $\mu^{n+1}(i)$ as above gives a representation that is very similar in appearance to the definition of $A_{n+1}$. That is,
    $$\mu^{n+1}(i)=\mu^n(i)\alpha(\mu^n(i))R(\beta(\mu^n(i))^R)\mu^n(i)R(\beta(\alpha(\mu^n(i)))^R)R(\beta(\mu^n(i))^R)\mu^n(i),$$
    whereas,
    $$A_{n+1}=(A_n, \alpha A_n, R\beta A_n, A_n, R\beta\alpha A_n, R \beta A_n, A_n).$$
    We can now show that the constructions are equivalent.

    \begin{lemma}
        The $S$-walk, $W_\mu$, produced by the morphism construction is identical to the $S$-walk, $W$, produced by the Gerver and Ramsey construction.
    \label{lem:equivalent-construction}
    \end{lemma}

    \begin{proof}
        For integers $n\geq 0$, let $A_n=(\mbf{a}_1, \dots, \mbf{a}_{7^n})$ as above in Section~\ref{sec:GerverRamseyConstruction}, where $\mbf{a}_p\in\{\mbf{i}, \mbf{j}, \mbf{k}\}$ for each integer $1\leq p\leq 7^n$ and $\mbf{a}_1=\mbf{i}$. From the definition of $A_{n+1}$, for $7^n+1\leq p \leq 7^{n+1}$ and $q=((p-1)\bmod{7^n}) + 1$ and $s=\lfloor\frac{p-1}{7^n}\rfloor$ we can write:
        \begin{equation}
            \mbf{a}_p =
            \begin{cases}
                \mbf{a}_q & \text{if } s\in\{0,3,6\},\\
                \alpha \mbf{a}_q & \text{if } s\in\{1\},\\
                \beta \mbf{a}_{7^n-q+1} & \text{if } s\in\{2,5\},\\
                \beta\alpha \mbf{a}_{7^n-q+1} & \text{if } s\in\{4\},
            \end{cases}
            \label{eqn:a}
        \end{equation}
        where $\alpha$ and $\beta$ are the vector operators defined in Section~\ref{sec:GerverRamseyConstruction}. We define two additional sequences $(b_p)_{p\geq 0}$ and $(c_p)_{p\geq 0}$ that track the parity of the number reversals and the parity of the number of vector operators applied, respectively, to get the $p^{\text{th}}$ term of $A_n$. Again, from the definition of $A_{n+1}$ we get $b_1=0$ and, 

        \begin{equation}
            b_p =
            \begin{cases}
                b_q & \text{if } s\in\{0,1,3,6\},\\
                1 - b_{7^n - q + 1} & \text{if } s\in\{2,4,5\}.
            \end{cases}
            \label{eqn:b}
        \end{equation}
        Similarly, we have $c_1=0$ and,
        
        \begin{equation}
            c_p =
            \begin{cases}
                c_q & \text{if } s\in\{0,3,6\},\\
                1 - c_q & \text{if } s\in\{1\},\\
                1 - c_{7^n-q+1} & \text{if } s\in\{2,5\},\\
                c_{7^n-q+1} & \text{if } s\in\{4\}.
            \end{cases}
            \label{eqn:c}
        \end{equation}
        
        We proceed by induction on $n$ to show that for all positive integers $p$ satisfying $7^n+1\leq p\leq 7^{n+1}$ we have $g(\mbf{a}_p,b_p,c_p)=\mu^{n+1}(i)[p]$ and $\phi(\mu^{n+1}(i)[p])=\mbf{a}_p$.

        We take $n=0$ as the base case. A direct application of the definitions of $\mbf{a}_p,b_p,c_p$, $\mu(i)[p]$, $g$, and $\phi$ gives  $g(\mbf{a}_p,b_p,c_p)=\mu^{n+1}(i)[p]$ and $\phi(\mu^{n+1}(i)[p])=\mbf{a}_p$ for $7^0+1\leq p \leq 7^1$.

        Now assume for $n=m$, for all positive integers $p$ with $7^m+1\leq p\leq 7^{m+1}$ we have that $g(\mbf{a}_p,b_p,c_p)=\mu^{m+1}(i)[p]$ and $\phi(\mu^{m+1}(i)[p])=\mbf{a}_p$. Suppose we have $p$ satisfying $7^{m+1}+1\leq p\leq 7^{m+2}$. We consider four cases for $s=\lfloor\frac{p-1}{7^{m+1}}\rfloor$. In all cases, for a given value of $p$, we take $q=((p-1)\bmod{7^{m+1}})+1$.

        \noindent {\bf Case 1:} $s\in \{0,3,6\}$.

        By equations~\ref{eqn:a},~\ref{eqn:b}, and \ref{eqn:c} we have
        \begin{align*}
            \mbf{a}_p &= \mbf{a}_q, \\
            b_p &= b_q, \\
            c_p &= c_q.
        \end{align*}
        This means that $g(\mbf{a}_p,b_p,c_p)=g(\mbf{a}_q,b_q,c_q)$. Since $7^m+1\leq q\leq 7^{m+1}$, then by the inductive hypothesis, we have $g(\mbf{a}_q,b_q,c_q)=\mu^{m+1}(i)[q]$. By the definition of $\mu$ and the assumption that $s\in\{0,3,6\}$ we get $\mu^{m+2}(i)[p]=\mu^{m+1}(i)[q]$. So $g(\mbf{a}_p,b_p,c_p)=\mu^{m+2}(i)[p]$ and $\phi(\mu^{m+2}(i)[p])=\mbf{a}_p$.

        \noindent {\bf Case 2:} $s\in\{1\}$.
        
        By equations~\ref{eqn:a},~\ref{eqn:b}, and \ref{eqn:c} we have
        \begin{align*}
            \mbf{a}_p &= \alpha\mbf{a}_q, \\
            b_p &= b_q, \\
            c_p &= 1-c_q.
        \end{align*}
        This means that $g(\mbf{a}_p,b_p,c_p)=g(\alpha\mbf{a}_q,b_q, 1-c_q)$. Since $7^m+1\leq q\leq 7^{m+1}$, and $s\in\{1\}$ we get $\mu^{m+2}(i)[p]=\mu^{m+1}(j')[q]$. By equation~\ref{eqn:j'} we get $\mu^{m+1}(j')[q]=\alpha(\mu^{m+1}(i))[q]=\alpha(\mu^{m+1}(i)[q])$ and by the inductive hypothesis this is equal to $\alpha(g(\mbf{a}_q,b_q,c_q))=g(\alpha\mbf{a_q},b_q,1-c_q) = g(\mbf{a}_p,b_p, c_p)$. So $g(\mbf{a}_p,b_p,c_p)=\mu^{m+2}(i)[p]$ and $\phi(\mu^{m+2}(i)[p])=\mbf{a}_p$.

        \noindent {\bf Case 3:} $s\in\{2,5\}$.

        Let $t=7^{m+1}-q+1$. By equations~\ref{eqn:a},~\ref{eqn:b}, and \ref{eqn:c} we have
        \begin{align*}
            \mbf{a}_p &= \beta\mbf{a}_{7^{m+1}-q+1}=\beta\mbf{a}_t, \\
            b_p &= 1 - b_{7^{m+1} - q + 1}=1 - b_t, \\
            c_p &= 1 - c_{7^{m+1} - q + 1}=1 - c_t.
        \end{align*}
        This gives $g(\mbf{a}_p,b_p,c_p)=g(\beta\mbf{a}_{t},1 - b_{t},1 - c_{t})$. Since $7^m+1\leq q\leq 7^{m+1}$ and $s\in\{2,5\}$ we get $\mu^{m+2}(i)[p]=\mu^{m+1}(i_b')[q]$. By equation~\ref{eqn:ib'} we get $$\mu^{m+1}(i_b')[q]=R(\beta(\mu^{m+1}(i))^R)[q]=R(\beta(\mu^{m+1}(i)))[t]=R(\beta((\mu^{m+1}(i)[t])))$$. By the inductive hypothesis this is equal to $$R(\beta(g(\mbf{a}_{t}, b_{t}, c_{t})))=g(\beta\mbf{a}_{t},1 - b_{t},1 - c_{t})=g(\mbf{a}_p,b_p,c_p).$$ So $g(\mbf{a}_p,b_p,c_p)=\mu^{m+2}(i)[p]$ and $\phi(\mu^{m+2}(i)[p])=\mbf{a}_p$.

        \noindent {\bf Case 4:} $s\in\{4\}$.

        Let $t=7^{m+1}-q+1$. By equations~\ref{eqn:a},~\ref{eqn:b}, and \ref{eqn:c} we have
        \begin{align*}
            \mbf{a}_p &= \beta\alpha\mbf{a}_{7^{m+1}-q+1}=\beta\alpha\mbf{a}_t, \\
            b_p &= 1 - b_{7^{m+1} - q + 1}=1 - b_t, \\
            c_p &= c_{7^{m+1} - q + 1}= c_t.
        \end{align*}

        This gives $g(\mbf{a}_p,b_p,c_p)=g(\beta\alpha\mbf{a}_{t},1-b_{t}, c_{t})$. Since $7^m+1\leq q\leq 7^{m+1}$ and $s\in\{4\}$ we get $\mu^{m+2}(i)[p]=\mu^{m+1}(k_b)[q]$. By equation~\ref{eqn:kb} we get $$\mu^{m+1}(k_b)[q]=R(\beta(\alpha(\mu^{m+1}(i)))^R)[q]=R(\beta\alpha((\mu^{m+1}(i)[t]))).$$ By the inductive hypothesis this is equal to $$R(\beta\alpha((g(\mbf{a}_{t},b_{t}, c_{t}))))=g(\beta\alpha\mbf{a}_{t}, 1-b_{t}, c_{t})=g(\mbf{a}_p, b_p, c_p).$$ So $g(\mbf{a}_p,b_p,c_p)=\mu^{m+2}(i)[p]$ and $\phi(\mu^{m+2}(i)[p])=\mbf{a}_p$.

        \noindent This covers all cases for showing $g(\mbf{a}_p,b_p,c_p)=\mu^{m+1}(i)[p]$ for $p\geq 1$ and $\phi(\mu^{m+1}(i)[p])=\mbf{a}_p$.
    \end{proof}
    
    \section{Bounding the number of collinear points}\label{sec:collinearbound}
    
    We devote this section to proving Theorem~\ref{thm:collinearbound}. We follow a similar argument to that used in Gerver and Ramsey's proof of their Theorem~2  \cite{gerver1979certain}. However, we extend the methods by making use of computer checks to verify results that are impractical to accomplish by hand. The first use of computer checks is in establishing an upper bound on the smallest indices of the distinct contiguous subsequences, of specified lengths, of vectors in $(\mbf{v}_p)_{p\geq 0}$.
    
    \begin{lemma}
        Let $I(n)$ be the index of the last new subword of length $n$ in $\lambda.$ That is, the largest index $j$ such that the there does not exist $k<j$ with $\lambda[k:k+n-1]=\lambda[j:j+n-1]$. Then $I(1)=215$ and $I(2)=558$ and $I(n)\leq 7\cdot I(\lceil n/7 \rceil + 1)$ for $n\geq 3$.
    \end{lemma}
    \begin{proof}
        To show that $I(1)=215$ it suffices to compute the first $215$ terms and observe that index $215$ is the first occurrence of symbol $i_b$ and all other symbols in the alphabet appear before that. To show that $I(2)=558$, first observe that for all $a\in\Sigma$ we have that $\mu(a)=aua$ for some $u\in\Sigma^*$. This means that every subword of length 2 that appears in $\lambda$ is a subword of $\mu(b)$ for some $b\in\Sigma$. Enumerating all subwords of length 2 that appear in $\mu(b)$ for some $b\in\Sigma$ and constructing the first $559$ symbols of $\lambda$, we find that index $558$ is the first occurrence of $j_b'i_b$ and all other possible subwords of length 2 appear at an earlier index.
        
        Now suppose that $n\geq 3$. Suppose that $j$ is the index of the first occurrence of word $w$ of length $n$ in $\lambda$. Consider the word $u$ of length $\lceil n/7\rceil + 1$ starting at index $\lfloor j/7 \rfloor$ in $\lambda$, that is $u=\lambda[\lfloor j/7\rfloor:\lfloor j/7\rfloor + \lceil n/7 \rceil]$. We have that $w$ is a subword of $\mu(u)$. This is because $\lambda[7\lfloor j/7\rfloor : 7(\lfloor j/7\rfloor +\lceil n/7\rceil)+6]=\mu(\lambda[\lfloor j/7\rfloor:\lfloor j/7\rfloor + \lceil n/7 \rceil])$ by the definition of $\lambda$, and we have $7\lfloor j/7\rfloor \leq j$ and $j+n-1\leq 7(\lfloor j/7\rfloor +\lceil n/7\rceil)+6$. Since the first occurrence of $u$ is at an index at most $I(\lceil n/7\rceil + 1)$, then we have that the first occurrence of $w$ starts at an index at most $7\cdot I(\lceil n/7\rceil + 1)$.
    \end{proof}
    
    Observe that since $\lceil n/7 \rceil + 1<n$ for $n\geq 3$ we can compute the index of the last new subword of length $n$ in $\lambda$ via a recursive procedure that finds an upper bound for the index of the last new subword, using the index of the last new subword of length $\lceil n/7\rceil + 1$ and checking all subwords of length $n$ up to index $7\cdot I(\lceil n/7 \rceil + 1)$. This procedure is implemented as \texttt{IndexOfLastNewSubword} \cite{Lidbetter_Avoiding_Collinearity_2022}.
    
    Using this upper bound, we can show that for every consecutive $16807$ points in $(\mbf{z}_p)_{p\geq 0}$, there are at most $6$ are collinear points. Executing the \texttt{IndexOfLastNewSubword} routine, we find that the last new subword of length $16807$ occurs at index $9{\small,}375{\small,}904$. Exhaustively considering all possible lines between pairs of points in $(\mbf{z}_p)_{p\geq 0}$ and tracking the number of points found on each such line, we find that there are at most $6$ collinear points. This is asserted by compiling and running the \texttt{count-collinear} part of the project with Rust using, for example, \texttt{cargo run --release 9375904} from within the \texttt{count-collinear} sub-directory of the Avoiding Collinearity software project \cite{Lidbetter_Avoiding_Collinearity_2022}. However, doing so with a single invocation will take a while. The result was verified by the author over the course of two weeks, using approximately 2 years and 9 months of CPU time across many cloud compute servers running in parallel, processing independent chunks of the sequence. This part of the project is implemented in Rust for performance reasons and also for built-in support of lightweight $128$-bit integer types, which are used to give precise canonical representations of lines between all pairs of points without running into integer overflow. A data dump of the results of the computation can be found in the \texttt{collinearity\_data.csv} file in the root of the Avoiding Collinearity software project \cite{Lidbetter_Avoiding_Collinearity_2022}.
    
    \begin{lemma}\label{lem:localbound}
        In every $16807=7^5$ consecutive indices of $(\mbf{z}_p)_{p\geq 0}$ there are at most $6$ collinear points.
    \end{lemma}
    
    We now proceed following the method of Gerver and Ramsey \cite{gerver1979certain}. So define $C_n^0 = \{\mbf{z}_0, \dots, \mbf{z}_{7^n}\}$ and more generally let $C_n^m=\{\mbf{z}_{m7^n},\dots, \mbf{z}_{(m+1)7^n}\}$. We have that the projection of $C_n^0$ onto the plane perpendicular to $\mbf{i}+\mbf{j}+\mbf{k}$ lies within a trapezoid with base $4^n\gamma$, base angles $\pi/3$, and adjacent sides with length $4^n\gamma/3$, with $\mbf{z}_0$ and $\mbf{z}_{7^n}$ lying at extreme ends of the base. Such a trapezoid is referred to as a trapezoid of order $n$ and we identify the set of all points inside the trapezoid of order $n$ containing the projection of $C_n^m$ with the notation $T_n^m$. By definition of $A_{n+1}$ it follows that the seven trapezoids of order $n$ fit together within a trapezoid of order $n+1$, as illustrated in Figure~\ref{fig:trapezoid-n+1}.
    
    \begin{figure}[ht]
    \centering
    \includegraphics[width=\textwidth]{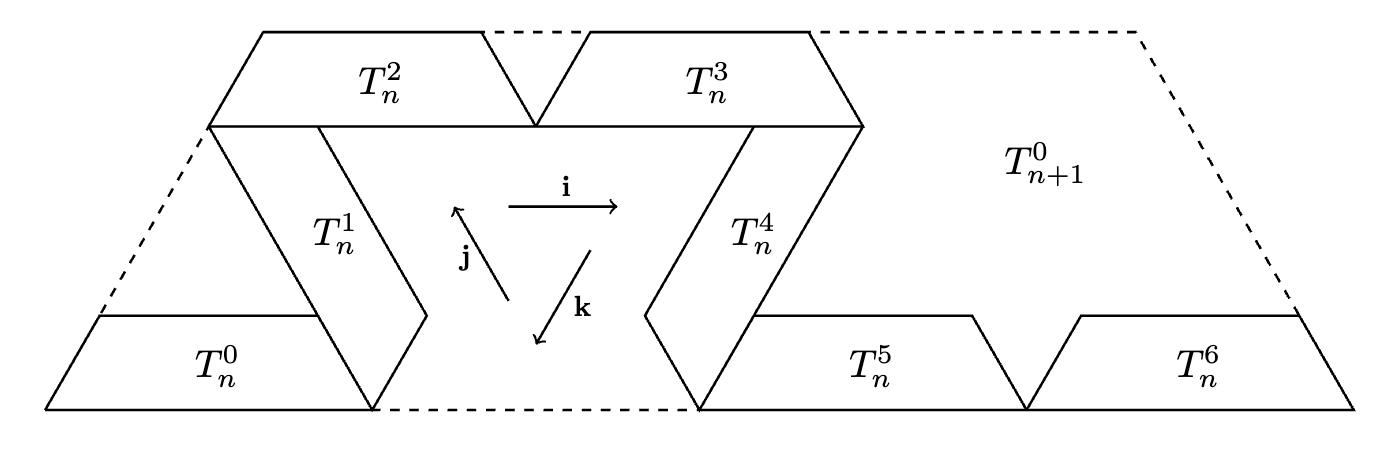}
    \caption{Trapezoids of order $n$ inside trapezoid of order $n+1$. Reproduced from Figure~1 of Gerver and Ramsey \cite{gerver1979certain}.}\label{fig:trapezoid-n+1}
    \end{figure}

    For a bigger picture view of how these trapezoids fit together, Figure~\ref{fig:trapezoid-multi-order} shows a trapezoid of order $n$ with $7$ trapezoids of order $n-1$ inside, each with $7$ trapezoids of order $n-2$ inside, each with $7$ trapezoids of order $n-3$ inside. This particular arrangement also corresponds to trapezoids $T_3^0$, and $T_2^0,\dots,T_2^{6}$, and $T_1^0,\dots, T_1^{48}$, and  $T_0^0,\dots ,T_0^{342}$. This figure is produced by the command:
    
    \texttt{DrawTrapezoids double 343 /home/finn/traps.png --recursive}.
    
    \begin{figure}[ht]
    \centering
    \includegraphics[width=\textwidth]{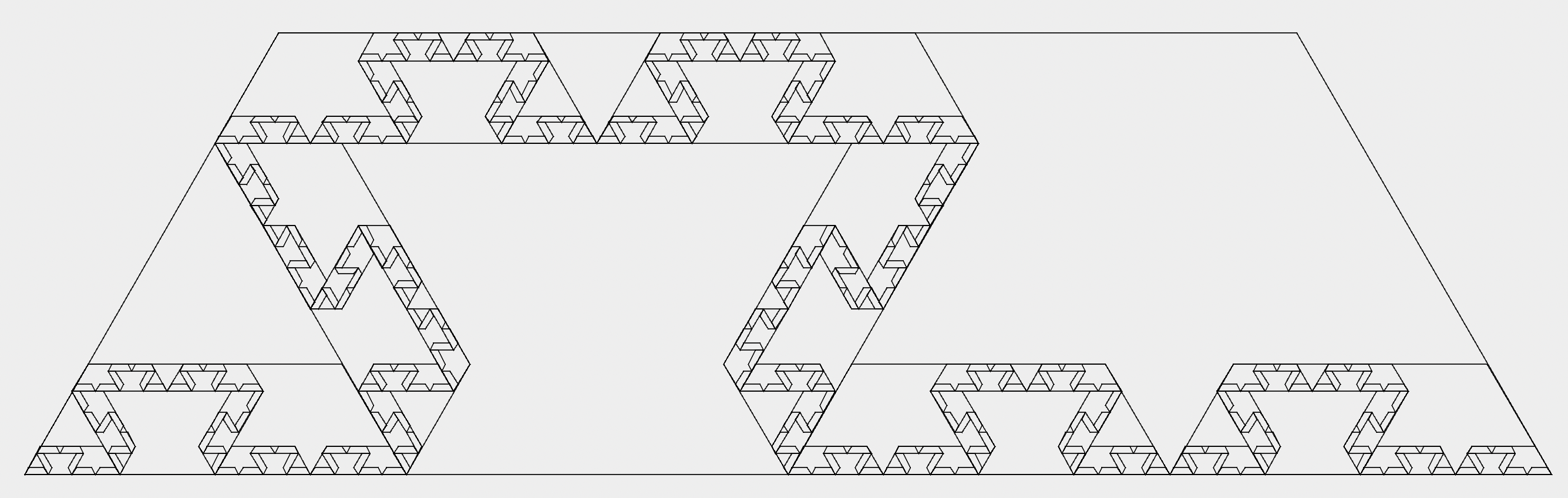}
    \caption{Example trapezoids of order $n$ through $n-3$.}\label{fig:trapezoid-multi-order}
    \end{figure}
    
    We refer to the $6$ different orientations of the trapezoids with $a,b,c,d,e,f$. Using the output map $\psi:\Sigma\rightarrow \{a,b,c,d,e,f\}$,
    \begin{align*}
        \psi(i)&=\psi(i_b')=a\\
        \psi(i')&=\psi(i_b)=b\\
        \psi(j)&=\psi(j_b')=c\\
        \psi(j')&=\psi(j_b)=d\\
        \psi(k)&=\psi(k_b')=e\\
        \psi(k')&=\psi(k_b)=f,
    \end{align*}
    Figure~\ref{fig:orientations} shows these orientations. We can show that trapezoid $T_n^m$ has orientation $\psi(\lambda[m])$.
    
    \begin{figure}[ht]
    \centering
    \includegraphics[width=\textwidth]{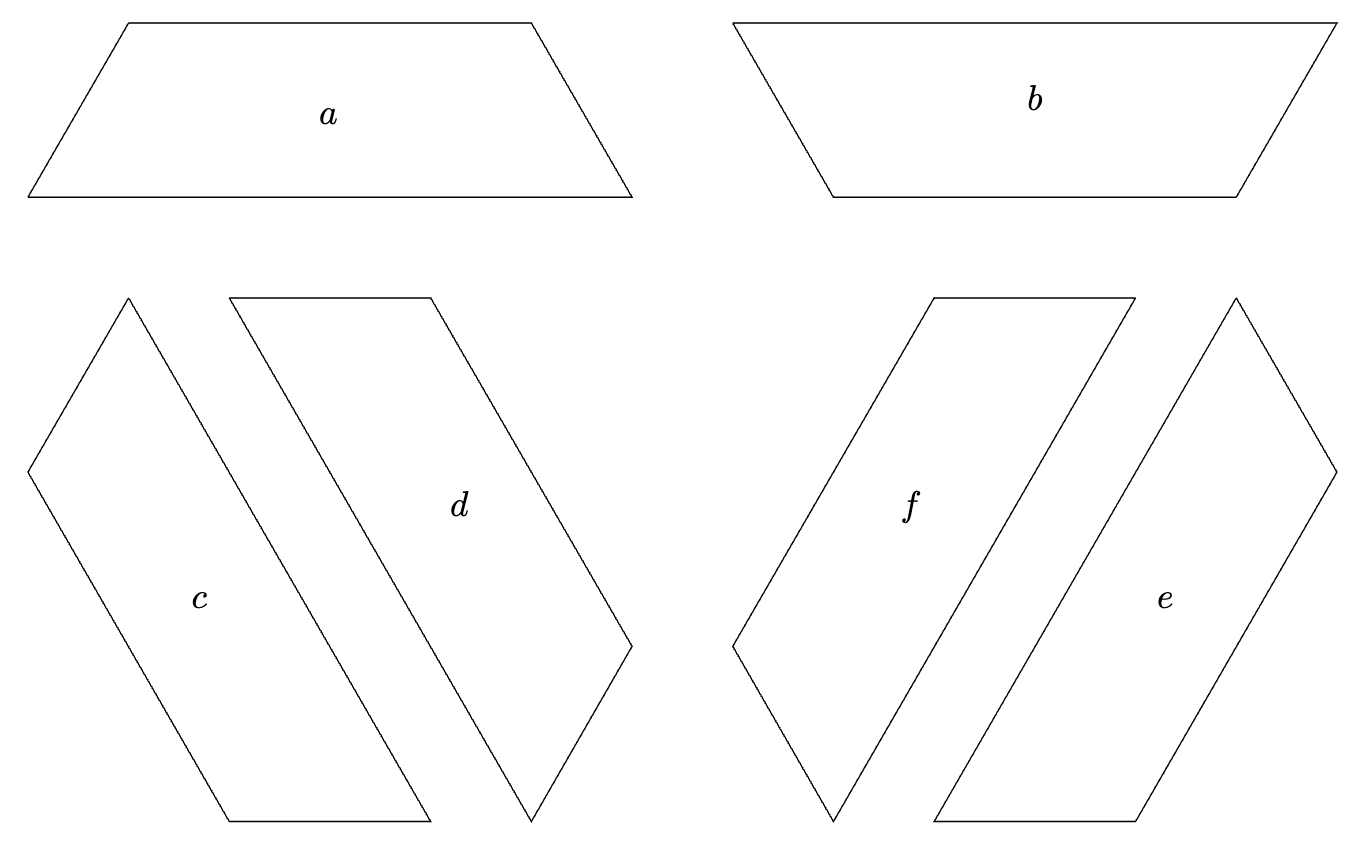}
    \caption{Trapezoid orientations for symbols $a,b,c,d,e,f$.}\label{fig:orientations}
    \end{figure}
    
    \begin{lemma}
        The orientation of trapezoid $T_n^m$ is given by $\psi(\lambda[m])$.
        \label{lem:trapezoid-orientation}
    \end{lemma}
    \begin{proof}
        We prove this via induction on $n$. For $n=0$ observe that the start and end points of the projection of $\mbf{i}$ onto the plane perpendicular to $\mbf{i}+\mbf{j}+\mbf{k}$ fits within both trapezoids $a$ and $b$. A similar statement holds for $\mbf{j},\mbf{k}$ and trapezoids $c,d$ and $e,f$ respectively. Choosing the orientation $\psi(\lambda[m])$ is compatible with this, so the base case holds for all $m$. For the inductive step we can consider each $\sigma\in\{i,j,k,i',j',k',i_b,j_b,k_b,i_b',j_b',k_b'\}$ and observe that the sequence of trapezoids given by $\psi(\mu(\sigma))$ fits within a larger trapezoid with orientation $\psi(\sigma)$, by drawing out each of these cases. Thus trapezoid $T_{n+1}^m$ has the same orientation as trapezoid $T_n^m$.
    \end{proof}
    
    For two points $p_1,p_2$ lying in a plane let $d(p_1,p_2)$ be the euclidean distance between $p_1$ and $p_2$. For two trapezoids of the same order, $T_n^a$ and $T_n^b$, we define the minimum distance between them as $d(T_n^a,T_n^b)=\min\{d(p_a,p_b) : p_a\in T_n^a \text{ and }p_b\in T_n^b\}$ and the maximum distance between them as $D(T_n^a,T_n^b)=\max\{d(p_a,p_b) : p_a\in T_n^a \text{ and }p_b\in T_n^b\}$. 

    Let $p,q,r,s,n,m$ be positive integers such that $1\leq n\leq m$ and $7^n\leq |p-q| < 7^{n+1}$ and $7^m\leq |r-s|<7^{m+1}$. We show that for all such $m,n,p,q,r,s$ where the points given by $\mbf{z}_p,\mbf{z}_q,\mbf{z}_r,$ and $\mbf{z}_s$ are collinear, we have $m-n\leq 3$. The case where $n=0$ is handled separately. 
    
    Consider each $c\in\{7,8,\dots, 48\}$ and suppose that $\frac{c7^n}{7}\leq |p-q| < \frac{(c+1)7^n}{7}$. If we assume, without loss of generality, that $p<q$, then if the projection of $\mbf{z}_p$ lies in $T_{n-1}^k$, then the projection of $\mbf{z}_q$ lies in either $T_{n-1}^{k+c}$ or $T_{n-1}^{k+c+1}$. 
    
    We can then define upper and lower bounds for $\norm{\mbf{z}_p - \mbf{z}_q}^\perp$ relative to $c$ and $n$.
    \begin{align}
      \ell(n,c)&=\min_{k\in \mathbb{N}}\{d(T_{n-1}^{k},T_{n-1}^{k+c}), d(T_{n-1}^{k},T_{n-1}^{k+c+1})\},\\
      h(n,c)&=\max_{k\in \mathbb{N}}\{D(T_{n-1}^{k},T_{n-1}^{k+c}), D(T_{n-1}^{k},T_{n-1}^{k+c+1})\}.
    \end{align}
    
    Then we have that if $\frac{c\cdot 7^n}{7}\leq |p-q| < \frac{(c+1)7^n}{7}$, then $\ell(n,c)\leq \norm{\mbf{z}_p - \mbf{z}_q}^\perp \leq h(n,c)$.
    
    To compute $\ell(n,c)$ and $h(n,c)$, first observe that since trapezoids of order $n$ are congruent to trapezoids of order $n-1$ and the base length of a trapezoid of order $n$ is four times that of a trapezoid of order $n-1$. This gives us $\ell(n,c)=4^{n-1}\cdot \ell(1,c)$ and $h(n,c)=4^{n-1}\cdot h(1,c)$.
    
    \begin{lemma}
        For positive integers $j,k$, and integer $m\geq 0$ if $\lambda[j:j+m]=\lambda[k:k+m]$, then $d(T_0^j,T_0^{j+m})=d(T_0^k,T_0^{k+m})$ and $D(T_0^j,T_0^{j+m})=D(T_0^k,T_0^{k+m})$.
    \end{lemma}
    \begin{proof}
        The orientations of the trapezoids of order $0$ are uniquely determined by the symbols in $\lambda$, by Lemma~\ref{lem:trapezoid-orientation}. Then since the sequence $(\mbf{v}_p)_{p>=1}$ between index $j$ and index $j+m$ is equal to the same as that between indices $k$ and $k+m$, trapezoids $T_0^k,\dots, T_0^{k+m}$ are translations of trapezoids $T_0^j,\dots,T_0^{j+m}$, respectively, by the same vector. Translations preserve distances, so $d(T_0^j,T_0^{j+m})=d(T_0^k,T_0^{k+m})$ and $D(T_0^j,T_0^{j+m})=D(T_0^k,T_0^{k+m})$.
    \end{proof}
    
    Hence, to compute $\ell(n,c)$ and $h(n,c)$ we only need to evaluate distances of pairs of trapezoids of order $0$ corresponding to endpoints of distinct subwords of length $c+2$ of $\lambda$. Thus, we can write
    
    \begin{align}
      \ell(1,c)&=4^{n-1}\cdot\min_{0<=k<=I(c+2)}\{d(T_{0}^{k},T_{0}^{k+c}), d(T_{0}^{k},T_{0}^{k+c+1})\},\label{eqn:l1c}\\
      h(1,c)&=4^{n-1}\cdot\max_{0<=k<=I(c+2)}\{D(T_{0}^{k},T_{0}^{k+c}), D(T_{0}^{k},T_{0}^{k+c+1})\}.\label{eqn:h1c}
    \end{align}
    
    Now suppose that $\frac{d\cdot 7^m}{7}\leq |r-s|< \frac{(d+1)7^m}{7}$, where $d\in\{7,\dots,48\}$. Then we have that $\ell(m,d)\leq \norm{\mbf{z}_r - \mbf{z}_s}^\perp \leq h(m,d)$. Suppose that $\mbf{z}_p,\mbf{z}_q, \mbf{z}_r,\mbf{z}_s$ are collinear. Then we must have that 
    $$\frac{\norm{\mbf{z}_p - \mbf{z}_q}^\perp}{\norm{\mbf{z}_p - \mbf{z}_q}^\parallel}=\frac{\norm{\mbf{z}_r - \mbf{z}_s}^\perp}{\norm{\mbf{z}_r - \mbf{z}_s}^\parallel},$$
    and therefore,
    $$\frac{\norm{\mbf{z}_p - \mbf{z}_q}^\perp}{|p - q|}=\frac{\norm{\mbf{z}_r - \mbf{z}_s}^\perp}{|r - s|}.$$
    We have 
    \begin{align}
      \frac{7\ell(n,c)}{(c+1)7^n}&\leq \frac{\norm{\mbf{z}_p - \mbf{z}_q}^\perp}{|p - q|} < \frac{7h(n,c)}{c\cdot 7^n},\label{eq:nc}\\
      \frac{7\ell(m,d)}{(d+1)7^m}&\leq \frac{\norm{\mbf{z}_r - \mbf{z}_s}^\perp}{|r - s|} <\frac{7h(m,d)}{d\cdot 7^m}.\label{eq:md}
    \end{align}
    Taking the leftmost expression in \ref{eq:nc} and the rightmost expression in \ref{eq:md} we get 
    \begin{align}
      \frac{7\ell(n,c)}{(c+1)7^n}     & < \frac{7h(m,d)}{d\cdot 7^m}\\
      \frac{7^m}{7^n}                 & < \frac{(c+1)h(m,d)}{d\cdot\ell(n,c)}\\
      7^{m-n}                         & < \frac{(c+1)4^{m-n}h(n,d)}{d\cdot\ell(n,c)}\\
      \left(\frac{7}{4}\right)^{m-n}  & < \frac{(c+1)h(n,d)}{d\cdot\ell(n,c)}\\
                                      & = \frac{(c+1)h(1,d)}{d\cdot\ell(1,c)}.
    \end{align}
    By considering all pairs $c,d\in\{7,\dots,48\}$, we find that the largest value for $$\frac{(c+1)h(1,d)}{d\cdot\ell(1,c)}$$ 
    is 
    
    $$\frac{(9+1)\sqrt{964}}{7\sqrt{28}}\approx 8.38226643996,$$
    and thus, certainly less than 9. This is asserted using the command
    
    \texttt{AssertBoundedDistanceRatio~7~48~wholeAndRt3~9~0},
    
    \noindent from the Java implementation \cite{Lidbetter_Avoiding_Collinearity_2022}. So we have $(7/4)^{m-n}<9<(7/4)^4\approx 9.379$. So $m-n<4$, but since $m$ and $n$ are integers, we know that $m-n\leq 3$. 
    
    This implies that there are at most $7^4$ collinear points in $W$. We can see this by assuming that $W$ has more than $7^4$ collinear points and deriving a contradiction. Let $p,q$ be the indices of points in $W$ that minimize $|p-q|$. Then there exist indices $r,s$ of points in $W$ such that $7^4|p-q|\leq |r-s|$. Let $n$ be an integer such that $7^n\leq |p-q|<7^{n+1}$, and let $m$ be an integer such that $7^m\leq |r-s|<7^{m+1}$. Then we have $7^{n+4}<7^{m+1}$, giving $m-n>3$. If $n>0$ then this is a contradiction. 
    
    But we can refine this argument further. Suppose that $X$ is a set of at least 2 collinear points in $W$. Let $t$ be the least integer, such that for all $\mbf{z}_r,\mbf{z}_s\in X$ we have $|r-s|<7^t$. Then there exists $\mbf{z}_p,\mbf{z}_q\in X$ such that $7^{t-1}\leq |p-q|<7^t$. If $t\leq 5$, then by Lemma~\ref{lem:localbound} there are at most $7^5\cdot \frac{6}{7^5}=6$ collinear points in $X$. If $t>5$, then we know that there are no two points of $X$ whose projections lie within the same trapezoid of order $n-4$. If this were the case, then there would exist points $\mbf{z}_r,\mbf{z}_s$ such that $7^{t-5}\leq |r-s|<7^{t-4}$. But upon comparing exponents, we have $(t-1)-(t-5)=4$, which violates the $m-n\leq 3$ inequality derived above.
     
    So, if we examine all ways that trapezoids of order $n-4$ can be arranged in the same trapezoid of order $n$ or adjacent trapezoids of order $n$ and find the maximum number of trapezoids of order $n-4$ that can be intersected by a single straight line, we will have a bound on the maximum number of collinear points in $X$. Note that when adjacent trapezoids of order $n$ are considered, we know that the difference between the index of the first trapezoid of order $n-4$ intersected by the straight line and the index of the last trapezoid of order $n-4$ intersected by the straight line is at most $7^4$, else this would correspond to having collinear points $\mbf{z}_p,\mbf{z}_q$ with $|p-q|\geq 7^{n+1}$.
     
    By Lemma~\ref{lem:trapezoid-orientation}, the sequence of trapezoid orientations is independent of their order, $n$. Thus we can check all possible configurations of $7^4$ consecutive trapezoids of order $0$ and count the maximum number of trapezoids intersected by a straight line.
     
    Exhaustively checking all possible configurations yields an upper bound of at most 188 trapezoids of order $n-4$ intersected by a single line. This is asserted using the Java implementation \cite{Lidbetter_Avoiding_Collinearity_2022} with the command
     
    \texttt{CountCollinearTrapezoids 2401 wholeAndRt3}.
     
    Therefore, $|X|\leq 188$ and there are no 189 collinear points in $W$, in the case where $n>0$.
    
    Now, we return to the case where $n=0$. That is $7^0=1\leq |p-q|< 7=7^1$. To handle this case we give separate definitions for $\ell(0,c)$ and $h(0,c)$, for $c\in\{1,\dots,6\}$ that give tighter bounds on the distances. Instead of working with trapezoids, we work with points directly. So we take $\ell(0,c)$ (respectively $h(0,c)$) to be the smallest (resp., largest) possible value for $\norm{\mbf{z}_p - \mbf{z}_q}^\perp$ such that $|p-q|=c$. These values are given in Table \ref{tab:n=0}.
    \begin{table}[H]
        \centering
        \begin{tabular}{|c|c|c|}
            $|p-q|$ & $\ell(0,|p-q|)$ & $h(0,|p-q|)$\\
            \hline
            1 & $\gamma$ & $\gamma$\\
            2 & $\gamma$ & $2\gamma$\\
            3 & $\sqrt{3}\gamma = \sqrt{2}$ & $\sqrt{3}\gamma = \sqrt{2}$\\
            4 & $\gamma$ & $\sqrt{7}\gamma$\\
            5 & $\gamma$ & $\sqrt{13}\gamma$\\
            6 & $\sqrt{3}\gamma = \sqrt{2}$ & $2\sqrt{3}\gamma = 2\sqrt{2}$ 
        \end{tabular}
        \caption{Values for $\ell(0,|p-q|)$ and $h(0,|p-q|)$ where $1\leq |p-q|<7$.}\label{tab:n=0}
    \end{table}
    
    Suppose that $7^m\leq |r-s|<7^{m+1}$ and $\mbf{z}_p$, $\mbf{z}_q$, $\mbf{z}_r$, $\mbf{z}_s$ are collinear. So there is some $d\in\{7,\dots, 48\}$ such that $\frac{d\cdot 7^m}{7}\leq |r-s| < \frac{(d+1)7^m}{7}$. By the collinearity of $\mbf{z}_p,\mbf{z}_q,\mbf{z}_r,\mbf{z}_s$ we have 
     
    \begin{align*}
        \frac{\norm{\mbf{z}_p - \mbf{z}_q}^\perp}{|p - q|}= \frac{\norm{\mbf{z}_p - \mbf{z}_q}^\perp}{\norm{\mbf{z}_p - \mbf{z}_q}^\parallel}=\frac{\norm{\mbf{z}_r - \mbf{z}_s}^\perp}{\norm{\mbf{z}_r - \mbf{z}_s}^\parallel}= \frac{\norm{\mbf{z}_r - \mbf{z}_s}^\perp}{|r - s|}.
    \end{align*}
    We can bound the leftmost and rightmost expressions as follows:
    \begin{align*}
        \frac{\ell(0,|p-q|)}{|p-q|}&\leq \frac{\norm{\mbf{z}_p - \mbf{z}_q}^\perp}{|p - q|} \leq \frac{h(0,|p-q|)}{|p-q|},\\
        \frac{7\cdot \ell(m,d)}{(d+1)\cdot 7^m}&\leq \frac{\norm{\mbf{z}_r - \mbf{z}_s}^\perp}{|r - s|} <\frac{7\cdot h(m,d)}{d\cdot 7^m}.
    \end{align*}
    Then we get 
    \begin{align*}
        \frac{\ell(0,|p-q|)}{|p-q|} &< \frac{7\cdot h(m,d)}{d\cdot 7^m},\\
        7^{m-1} &< \frac{|p-q|\cdot h(m,d)}{\ell(0,|p-q|)\cdot d} \leq \frac{5\cdot h(m,d)}{\gamma \cdot d} = \frac{5\cdot 4^{m-1}\cdot h(1,d)}{\gamma \cdot d},\\
        \left(\frac{7}{4}\right)^{m-1} &< \frac{5\cdot h(1,d)}{\gamma\cdot d}.
    \end{align*}
    Computing all values for $\frac{h(1,d)}{d}$, where $d\in\{7,\dots, 48\}$, we find that for all $d$ we have $\frac{5\cdot h(1,d)}{\gamma \cdot d}\leq \frac{10\sqrt{964}}{21}<14.89<\left(\frac{7}{4}\right)^5 = 16.4130859375$.
    This is asserted with the Java command \cite{Lidbetter_Avoiding_Collinearity_2022}
     
    \texttt{AssertBoundedMaxDistance 7 48 wholeAndRt3 1 2}.
     
    This command operates on scaled up trapezoids whose bases have length $6$, rather than length $4\gamma$. Furthermore, the command does not take the $\frac{5}{\gamma}$ multiplicative factor into account. Nevertheless, we can choose the bounding value of $1 + 2\sqrt{3}$ and take the necessary scaling adjustments into account to observe that
     
    $$(1+2\sqrt{3}) \cdot \frac{5\cdot 4\cdot \gamma}{6\cdot \gamma}<14.89.$$
    The reasons for using base lengths of $6$ in the implementation are described in Section~\ref{sec:algorithms}.
     
    So $m-1<5$, giving $m-1\leq 4$ and thus $m\leq 5$. So in this case we have $m-n\leq 5$, giving at most $7^6$ collinear points in $W$. However, by Lemma~\ref{lem:localbound} we have that in every $16807=7^5$ consecutive indices there are at most $6$ collinear points. If there are more than $7^6\cdot \frac{6}{7^5}$ collinear points then there must be some pair of indices $r,s$ of points in $W$ such that $7^6\leq|r-s|<7^7$. But this would violate the $m-n\leq 5$ inequality, since we assumed there are points with indices $p,q$ satisfying $1\leq |p-q|<7$. Therefore, in this case we have that there are at most $7^6\cdot \frac{6}{7^5}=42$ collinear points in $W$.
    
    Considering all cases we therefore have at most $188$ collinear points in $W$ and that there are no $189$ collinear points in $(\mbf{z}_p)_{p\geq 0}$. This concludes the proof of Theorem~\ref{thm:collinearbound}.
    
    \section{Algorithms and implementation}\label{sec:algorithms}
    We make some general comments on the approach taken for the implementation of the various commands, before describing the some specifics for computing distance ratios and counting collinear trapezoids. 
    
    The code for the Java commands \cite{Lidbetter_Avoiding_Collinearity_2022} has been written in a way that favours improving confidence in correctness, over efficiency and performance. In particular, the geometry calculations have been written to work with multiple ways of representing numbers. Since the points being considered have irrational coordinates, a number representation is made available that carries the radicals throughout the computation and (if necessary) only evaluates them at the end by using an over or under approximation, as appropriate, with a rational fraction. In order to give confidence in the correctness of the implementation, many unit tests are included in the project and multiple number systems are implemented for comparison. That is, in addition to the precise representation, support for a floating point representation is implemented for use in the same algorithm implementations by abstracting away the number system being used. Furthermore, in the precise representation, all arithmetic operations are checked for integer overflow. The program execution halts when using the precise representation if an operation results in integer overflow.
    
    For the precise representation, it was chosen to scale up the size of the trapezoids by a factor of $6/(4\gamma)$ to give the smallest trapezoids a base of length $6$ instead of length $4\gamma$. By doing this, all required coordinates are represented as sums of integers and integer multiples of $\sqrt{3}$. This is the reason for the use of the \texttt{WholeAndRt3} representation in the implementation. For consistency, the \texttt{DoubleRep} number system also uses trapezoids with the same dimensions. The \texttt{WholeNumber} number system cannot represent the trapezoids with the same dimensions precisely, but it can represent the trapezoid sequences in a way that preserves line intersections and collinearity by compressing all $y$ coordinates by a factor of $\sqrt{3}$. This allows it to be used for counting collinear trapezoids but not the distance ratios. On the other hand, the \texttt{DoubleRep} representation is not sufficiently precise to be able to work with the method for counting collinear trapezoids, but it can be used to assert bounds on the distance ratios. Some numerical analysis would be required to show that the operations do not introduce too much error. Such numerical analysis is unnecessary for the \texttt{WholeAndRt3} commands, which is why results are stated in Section~\ref{sec:collinearbound} using \texttt{WholeAndRt3} commands in favour of \texttt{DoubleRep}.

    On the other hand, the Rust \texttt{count\_collinear} command for counting the number of collinear points in contiguous subsequences of $W$ has been written for performance. The algorithm is implemented in the Rust file at path \texttt{cargo/count-collinear/src/compute.rs}  \cite{Lidbetter_Avoiding_Collinearity_2022}. The remainder of the code in the \texttt{cargo} project is for coordinating parallelising the computation. 
    
    \subsection{Computing distance ratios}
    
    Three commands are implemented for asserting bounds on distance ratios. These commands together assert upper bounds on each of $\frac{c+1}{\ell(1,c)}$, $\frac{h(1,d)}{d}$, and $\frac{(c+1)h(1,d)}{d\cdot \ell(1,c)}$ for a single provided range of integer values for $c$, for $d$, or for both $c$ and $d$. To find the maximum value for $\frac{(c+1)h(1,d)}{d\cdot \ell(1,c)}$, for $i\leq c,d \leq j$, we can independently maximize each of $\frac{(c+1)}{\ell(1,c)}$ and $\frac{h(1,d)}{d}$ and take the product of these maximum values.
    
    We compute $\ell(1,c)$ and $h(1,d)$ using the definitions given above in equations \ref{eqn:l1c} and \ref{eqn:h1c}, respectively, in Section~\ref{sec:collinearbound}. To do this we first compute the maximum index for which it is necessary to check for distances between trapezoids using the procedure described at the beginning of Section~\ref{sec:collinearbound}. Then we simply iterate over all trapezoid indices up to that maximum index and compute the minimum or maximum distances between points in trapezoids separated by the required number of indices, given by $c$ or $d$. The minimum distance between a pair of points in distinct trapezoids is found by considering each trapezoid vertex and finding the shortest distance to each trapezoid edge in the other trapezoid and taking the minimum of these values. The maximum distance between a pair of points in distinct trapezoids is found by considering only pairs of trapezoid vertices and taking the largest of these vertex-vertex distances.
    
    To avoid introducing precision errors, the square roots that would be introduced by euclidean distance computations are not evaluated. All required comparisons instead work with the squares of the distances.

    \subsection{Counting collinear trapezoids}
    
    To compute the largest number of trapezoids, separated by at most $2401$ indices, intersected by a single straight line, we consider each vertex of each trapezoid corresponding to each symbol of the distinct subwords of length $2401$ of $\lambda$ as a {\it pivot} vertex, and use a radial line sweep approach. For a given {\it pivot} vertex we first identify all trapezoids within $2401$ indices of the trapezoid of the pivot vertex's trapezoid. We then imagine continuously sweeping a half-infinite line from the pivot vertex from a starting position through a full rotation around the {\it pivot} vertex, identifying {\it enter} and {\it exit} event vertices for each of the trapezoids corresponding to when the sweep line first intersects a trapezoid and when it stops intersecting a trapezoid. To do this, for each trapezoid we sort its 4 vertices by angle relative to the position of the pivot vertex and an initial sweep line position. The least vertex in this ordering is the {\it enter} vertex, and the greatest vertex in this ordering is the {\it exit} vertex. Unless, of course, the trapezoid intersects the initial sweep line. If the trapezoid contains the pivot vertex, then it is included in all counts for this particular pivot vertex. Whereas if the initial sweep line intersects the trapezoid we can identify the appropriate {\it enter} and {\it exit} vertices by instead sorting the points relative to the sweep line after rotating it through 180 degrees and include this trapezoid in the initial count of intersected trapezoids. After identifying all {\it enter} and {\it exit} vertices we can sort these all together relative to the initial position of the sweep line, breaking ties by sorting {\it enter} vertices before {\it exit} vertices, and iterate over them. For each {\it enter} vertex we increment a counter, and for each {\it exit} vertex we decrement a counter. However, since we want to find the maximum number of trapezoids that are separated by at most $2401$ indices and we need to work with up to $4803$ trapezoids for a single pivot vertex, we cannot use just a single counter. Instead, we use a segment tree data structure that maintains the number of ``active'' trapezoids in each interval of indices of size $2401$. See Perparata and Shamos \cite{Preparata_1985_ComputationalGeometry}, for example, for a detailed description of the data structure. Each leaf node of the segment tree corresponds to one interval of size $2401$ and in each node we store a counter. When we encounter an {\it enter} vertex in the trapezoid with index $i$ we increment the counter for each interval of size $2401$ that contain this index. Similarly, when we encounter an {\it exit} vertex for some trapezoid with index $i$, we decrement the counter of each interval of size $2401$ containing this index. After each increment operation we query the tree for the largest value stored and compare this to the maximum counter value encountered thus far, and update the maximum if necessary. The segment tree data structure allows for these operations to be completed efficiently.
    
    While the index of the last distinct subword of length $2401$ in $\lambda$ is $1{\small,}339{\small,}415$ we need only count the number of collinear trapezoids for indices corresponding to distinct subwords. But we can go further. First, observe that we only care about distinct sequences of trapezoids, rather than subwords in $\lambda$. Second, since collinearity is preserved by rotations and reflections we can restrict our search to distinct sequences of trapezoids of length $2401$ after normalizing to take these rotations and reflections into consideration. We can do this by identifying each of the $6$ possible single trapezoid orientations with the elements of the permutation group $S_3$ and normalize sequences of trapezoids to always start with the identity, by applying the group action corresponding to the inverse of the first element in a sequence of trapezoids. For trapezoid orientations $a,b,c,d,e,f$, we have the Cayley table given in Table~\ref{tab:cayley}.
    \begin{table}[H]
        \centering
        \begin{tabular}{c|c|c|c|c|c|c|}
                & $a$ & $b$ & $c$ & $d$ & $e$ & $f$\\
            \hline
            $a$ & $a$ & $b$ & $c$ & $d$ & $e$ & $f$ \\
            \hline
            $b$ & $b$ & $a$ & $d$ & $c$ & $f$ & $e$ \\
            \hline
            $c$ & $c$ & $d$ & $f$ & $e$ & $b$ & $a$ \\
            \hline
            $d$ & $d$ & $c$ & $e$ & $f$ & $a$ & $b$ \\
            \hline
            $e$ & $e$ & $f$ & $b$ & $a$ & $c$ & $d$ \\
            \hline
            $f$ & $f$ & $e$ & $a$ & $b$ & $d$ & $c $\\
            \hline
        \end{tabular}
        \caption{Cayley table for permutation group $S_3$ with elements corresponding to trapezoid orientations.}\label{tab:cayley}
    \end{table}
    So, for example, the subword $j'i_b'ik_bi_b'i$ corresponds to the trapezoid sequence $\psi(j'i_b'ik_bi_b'i)=daafaa$. The inverse to $d$ is $e$, so applying the group operation corresponding to $e$ to each element of the trapezoid sequence, we get the normalized representation $aeedee$. Restricting the search in this way allows the computation of the largest number of trapezoids separated by at most $2401$ indices and intersected by a single straight line to complete in approximately 30 minutes on a 2020 M1 MacBook Air running macOS version 11.6.
    
    For verification purposes, the output of this algorithm was compared to a much simpler implementation that considers all lines defined by pairs of trapezoid vertices and counts the number of trapezoids intersected by that line. Comparing the outputs on thousands of smaller cases yielded no discrepancies in the results of the two implementations.

    \subsection{Counting collinear points}
    The algorithm used for calculating the largest number of collinear points in the first $n$ terms of the $S$-walk works by iterating over all pairs of points and getting a canonical representation for the infinite line through those points. For each canonical line encountered, a counter is maintained for the number of pairs of points that lie on that canonical line. In the implementation it was found that using a hashmap data structure achieved higher performance than a sorted map. While iterating over all pairs of points, simply keep track of the current largest number of points that all lie on the same line. This value is reported after the iteration over all pairs of points terminates.
    
    \section{Further improving the bound}
    \label{sec:improving}
    
    Before we consider how to improve the upper bound, it is worth pointing out that for the $S$-walk considered here thus far, the number of collinear points is at least $6$. The first example of $6$ collinear points is:
    \begin{align*}
        (46, 40, 23) & \text{ at index } 109,\\
        (48, 41, 24) & \text{ at index } 113,\\
        (64, 49, 32) & \text{ at index } 145,\\
        (66, 50, 33) & \text{ at index } 149,\\
        (82, 58, 41) & \text{ at index } 181,\\
        (84, 59, 42) & \text{ at index } 185.
    \end{align*}
    Note that this is larger than the value of $3$ stated by Gerver and Ramsey in the last paragraph before the statement of their Theorem~3 \cite{gerver1979certain} as the likely true value for the number of collinear points in this $S$-walk. 
    
    To compute the result of Lemma~\ref{lem:localbound}, it was determined that in the first $10$ million indices there are no 7 collinear points. This seems like strong evidence that largest number of collinear points in this $S$-walk is 6. Actually demonstrating an upper bound of $6$ seems difficult using these methods, even with greater computational resources. For the $n=0$ case we have a bound of at most $42$ collinear points. However, improving upon the $n>0$ case would require improving the bound on the difference in exponents at least to $m-n\leq 2$. This makes the collinear trapezoid computation easier and it would yield a result of at most $62$ collinear points, using \texttt{CountCollinearTrapezoids 343 wholeAndRt3}. However, there is no reason to restrict the count of the number of intersected trapezoids to two dimensions. The points of the sequence could be considered in three dimensions and shown to lie in trapezoidal prisms of order $n$. 
    
    There is a well-studied problem for determining the existence of {\it stabbing lines} of a set of convex polyhedra; see, for example, \cite{agarwal1994stabbing,avis1988polyhedral, nielsen2000fast}. A stabbing line for a set of convex polyhedra is an infinite line that intersects at least one facet of each polyhedron in the set. Some of the algorithms defined for identifying stabbing lines can be adapted to find the lines that intersect the largest number of polyhedra from a set. For example, for a set of convex polyhedra with $n$ vertices in total, the algorithm described by Avis and Wenger \cite{avis1988polyhedral} enumerates $O(n^3)$ candidates for stabbing lines. The lines that maximise the number of polyhedra intersected is one of these stabbing line candidates. These candidate lines can be found in $O(n^3\log n)$ time. This gives an $O(n^4\log n)$ algorithm for finding the largest number of polyhedra from the set intersected by a single line, by simply trying each candidate line and counting how many polyhedra in the set are intersected. It is conceivable that this could be improved upon for the special case of the trapezoidal prisms involved. However, for trapezoidal prisms separated by at most $2401$ indices, even an $O(n^3)$ algorithm would require excessive computational resources. On the other hand, implementing an algorithm to count the number of trapezoidal prisms separated by at most $343$ indices for all possible sequences of $343$ trapezoidal prisms may be feasible. To get the $m-n\leq 2$ bound required to limit the search to $343$ consecutive trapezoids, we would need a sufficiently fine partition of $7^n$ through $7^{n+1}$ and $7^m$ through $7^{m+1}$ to assert that the distance ratio, as above, is less than $(7/4)^3=5.359375$. Using the same strategy as above, considering all pairs $c,d\in\{49,\dots,342\}$, we get a bound of $\frac{239\sqrt{14400}}{54\sqrt{7168}}\approx 6.27316$. Pushing this further to all pairs $c,d\in\{343,\dots,2400\}$, we get a bound of $\frac{1661\sqrt{236196}}{394\sqrt{115492}}\approx 6.02884$. With the existing approach and implementation, it was not feasible to go as far as $c,d\in\{2401,\dots,16806\}$, due to the large number of distinct trapezoid sequences of length $16807$ required to consider. 
    
    This still leaves a few natural open questions. First, can the upper bound be improved to at most $6$ collinear points for this particular sequence? Second, the result presented here is particular to one specific sequence. By choosing an alternative $S$-walk for $S\subset\mathbb{Z}^3$ can it be shown that there is an infinite $S$-walk with at most $k<6$ collinear points? And finally, if it cannot be shown for the $3$-dimensional case, then is there an $S$-walk with $S\subset \mathbb{Z}^n$ with no $3$ collinear points, and if so what is the least $n$ for which this is the case?
   
    \section*{Acknowledgments}
    I would like to thank Alexander Bailey, Clayton Goes, Bradley Kleiboer, and several other (anonymous) employees at RideCo for contributing to help cover the server costs used to prove Lemma~\ref{lem:localbound}. I would also like to thank Robin Lidbetter for lending physical computational resources for use towards proving Lemma~\ref{lem:localbound}. Finally, I wish to thank Luke Schaeffer and Jeffrey Shallit for their invaluable comments and suggestions.
    \printbibliography

\end{document}